\documentclass[11pt, a4paper]{amsart}
\usepackage[margin=3cm]{geometry}
\usepackage{amsthm}
\usepackage[colorlinks=true,citecolor=cyan]{hyperref}
\usepackage{amsmath}
\usepackage{amsfonts}
\usepackage{amssymb}
\usepackage{stmaryrd}  
\usepackage{mathrsfs} 
\usepackage{esint} 
\usepackage{array} 
\usepackage{latexsym}
\usepackage{lmodern} 
\usepackage{graphicx} 
\usepackage{graphics} 
\usepackage{epsfig} 
\usepackage{color} 
\usepackage[dvipsnames]{xcolor}
\usepackage[all]{xy}
\usepackage[new]{old-arrows}
\usepackage{textcomp}
\usepackage{stmaryrd}
\DeclareGraphicsExtensions{.eps,.pdf,.jpeg,.png} 
\usepackage{fancyhdr}
\usepackage{multirow} 
\usepackage{comment}
\usepackage{pgf,tikz}
\usepackage{paralist}
\usepackage{faktor}
\usepackage{dirtytalk}
\usepackage{url}
\usepackage{blkarray,bigstrut}
\usepackage{nicematrix}
\usepackage{enumerate}

\makeatletter
\renewcommand*\env@matrix[1][*\c@MaxMatrixCols c]{%
  \hskip -\arraycolsep
  \let\@ifnextchar\new@ifnextchar
  \array{#1}}
\makeatother

\newtheorem{proposition}{Proposition}[section]
\newtheorem{theorem}{Theorem}[section]
\newtheorem{lemma}{Lemma}[section]

\theoremstyle{definition}
\newtheorem{definition}{Definition}[section]
\newtheorem{remark}{Remark}[section]
\newtheorem{example}{Example}[section]

\title{Affine twisted length function}

\author{Nathan Chapelier-Laget}

\begin{document}

\maketitle

\begin{abstract}
Let $W_a$ be an affine Weyl group. In 1987 Jian Yi Shi gave a characterization of the elements  $w \in W_a$ in terms of $\Phi^+$-tuples $(k(w,\alpha))_{\alpha \in \Phi^+}$ called the Shi vectors. Using these coefficients, a formula is provided to compute the standard length of $W_a$. In this note we express the twisted affine length function of $W_a$ in terms of the Shi coefficients.
\end{abstract}

\section{Introduction}

Each Coxeter group $W$ has a natural generating set $S$: the one used in its presentation $(W,S)$. This set allows us to define the standard length function $\ell$. This function is fairly well understood and of great interest as it provides relations between the algebraic structure of reflection groups and their underlying geometry. It is used, among others, as an organizing principle to prove that any finite Euclidean reflection group is a Coxeter group. This function is also involved in a wide variety of   algebraic and combinatorial subjects such as Hecke algebras, Bruhat order, Kazhdan-Lusztig polymonials, coset representatives, descent algebras, inversion sets, Garside shadows, Artin-Tits monoids etc.

Another useful function obtained from $\ell$ and related to the twisted orders on $W$ is the twisted length function $\ell_A$ introduced in \cite{Shelling2}, where $A \in \mathcal{P}(T)$ and $T$ is the set of reflections of $W$.  A reflection order of $W$ is a  total order on $T$ satisfying a local condition on any dihedral reflection subgroup of $W$ (see \cite{Shelling1}, Definition 2.1). Initial sections of reflection orders, which are certain subsets $ A\subset T$, lead to define partial pre-orders $\leq_{A} $ on $W$ that are similar to Bruhat order. The twisted length function plays an important role to determine which are the subsets $A\subset T$ such that $\leq_A$ is a partial order \cite{Shelling2, edgar2007sets}. The collection of subsets of $T$ satisfying this property is exactly the set of biclosed subsets of $T$, which is denoted by $\mathcal{B}(T)$ (we recall this notion in Section \ref{general back}).

In the context of affine Weyl groups, length function has also been well studied. In particular Jian Yi Shi gave a geometrical interpretation of $\ell$ as a sum of absolute values of integers subject to certain conditions \cite{JYS1} (we recall this material in Section \ref{shi para}). These integers are called Shi coefficients. 
In private communications Matthew Dyer conjectured, in the setting of affine Weyl groups, that the length $\ell_A$ (with $A \in \mathcal{B}(T)$) should be related to the Shi coefficients. 

This note contains two results: Proposition \ref{prop} and Theorem \ref{theo gen}. These results provide geometric tools in order to determine the twisted length function. In particular, Proposition \ref{prop} expresses the affine twisted length function $\ell_A$ in terms of Shi coefficients when $A$ is a very specific subset of $T$. For this specific set $A$, the order $\leq_{A}$ is essentially Lusztig's alcoves order introduced in \cite{Lus} Sections 1.4, 1.5. Theorem \ref{theo gen} expresses $\ell_A$ for any $A \in \mathcal{B}(T)$ in terms of Shi coefficients. The formula in this theorem is given by means of a sum and difference of four terms that are introduced in Definition \ref{def}.

The organisation of this  paper is as follows. Section \ref{general back} recalls some general material about Coxeter groups, twisted order and twisted length function. Sections \ref{affine Weyl groups}, \ref{shi para} recall basic definitions about affine Weyl groups and particularly Section \ref{shi para}  explains the notion of Shi coefficients. Section \ref{affine root} gives a theorem (due to M. Dyer) used later on to prove Theorem \ref{theo gen}. Finally, in Section  \ref{results} we give and prove our two statements.
 
 \newpage

\subsection{General background}\label{general back} Let $(W,S)$ be a Coxeter system with standard length function $\ell$, root system $\Psi$ and positive root system $\Psi^+$. The length function satisfies: for all $w \in W$
\begin{align}\label{length function}
\ell(w) = \ell(w^{-1}).
\end{align}

 Let $T = \{wsw^{-1}~|~s \in S,~w \in W\}$ be the set of reflections of $W$. The set $\Psi^+$  is naturally in bijection with $T$ and by denoting $\chi$ this bijection, we set $s_{\alpha}:=\chi(\alpha)$ to be the corresponding reflection of  $\alpha \in \Psi^+$.

Let $\mathcal{P}(T)$ be the power set of $T$. This set  is an abelian group under symmetric difference: $A+B := (A \setminus B) \cup (B \setminus A)$ and $W$ acts on $\mathcal{P}(T)$ by conjugation.

For $w \in W$ we denote $N(w) = \{t \in T~|~\ell(tw) < \ell(w)\}$. The cardinality of $N(w)$ is equal to $\ell(w)$. The map $N : W \rightarrow \mathcal{P}(T)$ is a reflection cocycle, that is $ N(xy) = N(y) + xN(y)x^{-1}$ for all $x,y \in W$ and $N(s) =\{s\}$ for all $s \in S$.

Twisting the conjugation action on $\mathcal{P}(T)$ by $N$ we obtain another action, called the twisted conjugation, and defined by
$$
\begin{array}{ccc}
W \times \mathcal{P}(T) & \longrightarrow & \mathcal{P}(T) \\
     (w, A) & \longmapsto & w\cdot A:= N(w) + wAw^{-1}.
\end{array}
$$

Let $A \in \mathcal{P}(T)$. The twisted length function $\ell_A$ on $W$ is defined by
 $$
\ell_A(w):= \ell(w) -2|N(w^{-1}) \cap A|.
$$ 
For $x,y \in W$ the function $\ell_A$ has the property (which is proved in \cite{Shelling2}, Proposition 1.1\footnote{In this proposition M. Dyer took a specific $A$, but the formula is still valid for all $A \in \mathcal{P}(T)$ and the proof is exactly the same.})
\begin{align}\label{length}
\ell_{A}(xy) = \ell_{A}(y) +\ell_{y\cdot A}(x).
\end{align}

Let $X \subset \Psi^+$. We denote by $\text{cone}_{\Psi}(X) := \text{cone}(X) \cap \Psi$. We say that $X$ is closed if $\text{cone}_{\Psi}(\{\alpha,\beta\}) \subset X$ for all $\alpha, \beta \in X$. We also say that $X$ is biclosed if $X$ and $\Psi^+\setminus X$ are both closed. A set $A \in \mathcal{P}(T)$ is called closed (resp. biclosed) if $\chi^{-1}(A)$ is closed (resp. biclosed). We denote by $\mathcal{B}(T)$ the set of biclosed subsets of $T$.
Using the cocycle property of $N$ and Proposition 2.11 of \cite{ISWO} it is easy to see that the set $\mathcal{B}_f(T)$ of finite biclosed subsets of $T$ is stable under the $N$-twisted conjugation. The more general case is proved in \cite{dyer2019weak} Lemma 4.1

Let $A \in \mathcal{P}(T)$. Define now a pre-order $\leq_A$ on $\mathcal{P}(T)$ by: $v \leq_A w$ iff there exist $t_1,\dots, t_p \in T$ such that $v=t_p \cdots t_1w$ and $t_i \in t_i\cdots t_1wA$ for $i=1,\dots, p$.  In particular $\leq_{\emptyset}$ is the usual Bruhat order and $\leq_{T}$ is the reverse Bruhat order. The pre-order $\leq_A$ have been strongly investigated in recent decades \cite{Shelling1, Shelling2, dyer2019weak, edgar2007sets} and used, among others, to prove the existence of Kazhdan-Lusztig polynomials and to  study the structure constants for the generic Hecke algebra \cite{Shelling2}. 

There exists a characterisation of subsets $A \in \mathcal{P}(T$) such that $\leq_{A}$ is a partial order. M. Dyer showed that if $A \in \mathcal{B}(T)$ then $\leq_{A}$ is a partial order \cite{Shelling2}. The reverse direction was proved by T. Edgar \cite{edgar2007sets}.

For $A \in \mathcal{B}(T)$, the length $\ell_A$ can be used to characterize the partial order $\leq_A$: $v \leq_A w$ iff there exist $t_1,\dots, t_p \in T$ such that $w=t_p \cdots t_1v$ and $	\ell_A(t_{i-1}\cdots t_1v) \leq \ell_A(t_i\cdots t_1v)$ for $i=1,\dots, p$.

\subsection{Affine Weyl groups}\label{affine Weyl groups}

Let $V$ be a Euclidean space with inner product $\langle -, -\rangle$. We denote $||x|| = \sqrt{\langle x,x \rangle} $ for $x \in V$. Let $\Phi$ be an irreducible crystallographic root system in $V$ with simple system $\Delta:=\{\alpha_1,~\dots, \alpha_n\}$.  From now on, when we will say \say{root system} it will always mean irreducible  crystallographic root system. We denote by $\mathbb{Z}\Phi$ the $\mathbb{Z}$-lattice generated by $\Phi$ and we identify $\mathbb{Z}\Phi$ with the group of its associated translations.

Let $W$ be the \emph{Weyl group} associated to $\mathbb{Z}\Phi$, that is the maximal  reflection subgroup of $O(V)$ admitting $\mathbb{Z}\Phi$ as a $W$-equivariant lattice, or equivalently the group generated by the reflections associated to $\Delta$.

\begin{comment}
Let $\alpha \in \Phi$. We write 
$$
\begin{array}{ccccc}
s_{\alpha}  & : & V & \longrightarrow & V \\
                 &   & x & \longmapsto     & x-2\frac{\langle \alpha, x \rangle}{\langle \alpha, \alpha \rangle}\alpha.
\end{array}
$$

It is known that the Coxeter group associated to $\Phi$, i.e the subgroup of $O(V)$ generated by the reflections $s_{\alpha}$, is actually the Weyl group $W$. Each Weyl group has a structure of Coxeter group with set of Coxeter generators $S:=\{s_{\alpha_1},..,s_{\alpha_n}\}$.
\end{comment}
Due to the classification of irreducible crystallographic root systems, we know that there are at most two possible root lengths in $\Phi$. We call short root the shorter ones. 
Let $\alpha \in \Phi$ such that  $\alpha = a_1\alpha_1 + \cdots + a_n\alpha_n$ with $a_i \in \mathbb{Z}$. The height of $\alpha$ (with respect to $\Delta$) is defined by the number $h(\alpha) = a_1 + \cdots+ a_n$. Height provides a pre-order on $\Phi^+$ defined by $\alpha \leq \beta$ if and only if $h(\alpha) \leq h(\beta)$. We denote by $-\alpha_0$ the \emph{highest short root} of $\Phi$.

For $\alpha \in \Phi$ we write $\alpha^{\vee}:= \frac{2\alpha}{\langle \alpha, \alpha \rangle}$. Let $k \in \mathbb{Z}$. Define the affine reflection $s_{\alpha,k} \in \text {Aff}(V)$ by  $s_{\alpha,k}(x)= x-(\alpha^{\vee}\langle \alpha, x \rangle-k)\alpha.$ The affine Weyl group associated to $\Phi$, denoted $W_a$,  is defined by
$$
W_a = \langle s_{\alpha,k}~|~\alpha \in \Phi, ~k \in \mathbb{Z \rangle}.
$$

It is known that  $W_a \simeq  \mathbb{Z}\Phi\rtimes W $ (see for example \cite{Hum}, Ch 4). In particular any element $w \in W_a$ decomposes as $w=\tau_x\overline{w}$ where $\tau_x$ is the translation corresponding to $x \in \mathbb{Z}\Phi$ and $\overline{w} \in W$ is called the finite part of $w$. The pair $(W_a, S_a)$ is a Coxeter system where $ S_a := \{s_{\alpha_1},\dots,s_{\alpha_n}\} \cup \{s_{-\alpha_0,1}\}$. %For short we will write $S_a = \{ s_0, s_1, \dots s_n\}$ where $s_0 := s_{-\alpha_0,1}$ and $s_i = s_{\alpha_i}$ for $i =1,\dots, n$. 

\subsection{Shi parameterization}\label{shi para}

 For any $\alpha \in \Phi$, any $k \in \mathbb{Z}$ and any $m \in \mathbb{R}$, we define the hyperplanes 
$$
H_{\alpha,k} =  \{ x \in V~|~ \langle x, \alpha^{\vee} \rangle = k\},
$$
\noindent and the strips
\begin{align*}
H_{\alpha,k}^m & = \{x \in V~|~k < \langle x ,\alpha^{\vee} \rangle < k+m \}.
\end{align*}

We denote by $\mathcal{H}$ the set of all the hyperplanes $H_{\alpha,k}$ with $\alpha \in \Phi^+$, $k \in \mathbb{Z}$. It is easy to see that $H_{-\alpha,k} = H_{\alpha,-k}$. Therefore we need only to consider the hyperplanes $H_{\alpha,k}$ with $\alpha \in \Phi^+$ and $k \in \mathbb{Z}$.

 The connected components of 
$$
 V ~\backslash \bigcup\limits_{\begin{subarray}{c}
 ~ ~\alpha \in \Phi^{+} \\ 
  k \in \mathbb{Z}
\end{subarray}}
H_{\alpha,k} 
$$
are called \emph{alcoves}. We denote by $\mathcal{A}$ the set of all the alcoves and $A_e$ the alcove defined as $A_e = \bigcap_{\alpha \in \Phi^+} H_{\alpha,0}^1$. $W_a$  acts on $\mathcal{A}$ and it turns out that this action is regular (see for example \cite{Hum}, Ch 4). Thus, there is a bijective correspondence between  $W_a$ and $\mathcal{A}$.  This bijection is defined by $w \mapsto A_w$ where $A_w := wA_e$. We call $A_w$ the corresponding alcove associated to $w$. 

Any alcove of $V$ can be written as an intersection of particular strips, that is there exists a $\Phi^+$-tuple of integers $(k(w,\alpha))_{\alpha \in \Phi^+}$ such that 
$$
A_w = \bigcap\limits_{\alpha \in \Phi^+}H_{\alpha, k(w,\alpha)}^1.
$$

The coefficient $k(w,\alpha)$ is called the Shi coefficient of $w$ in position $\alpha$ and the vector $(k(w,\alpha))_{\alpha \in \Phi^+}$ is called the Shi vector of $w$.
For any $w \in W_a$ and any $\alpha \in \Phi$ we use the conventions 
\begin{align}\label{conv2}
k(w,-\alpha) = -k(w,\alpha),
\end{align}
\begin{align}\label{conv1}
k(w,\alpha^{\vee}) = k(w,\alpha).
\end{align}

 In the setting of affine Weyl groups the length of any element $w \in W_a$ is easy to compute via the coefficients $k(w,\alpha)$. Indeed, thanks to Proposition 4.3 in \cite{JYS1} we have 
  \begin{align}\label{long}
 \ell(w) = \sum\limits_{\alpha \in \Phi^+} |k(w,\alpha)|.
 \end{align}

\subsection{Affine root system}\label{affine root}

Let $\widehat{V} = V \oplus \mathbb{R}\delta$ with $\delta$ an indeterminate. The inner product $\langle -,- \rangle$ has a unique extension to a symmetric bilinear form on $\widehat{V}$ which is positive semidefinite and has a radical equal to the subspace $\mathbb{R}\delta$. This extension is also denoted $\langle -,- \rangle$, and it turns out that the set of isotropic vectors associated to the form $\langle -,- \rangle$ is exactly $\mathbb{R}\delta$.

The root system of $W_a$ is denoted by $\Phi_a$. Using \cite{mDgL}  (Section 3.3 Definition 4 and Proposition 2) a concrete description of the affine (resp. positive) root system of $W_a$ is provided by:
\begin{align*}
\begin{split}
\Phi_a     &= \Phi^{\vee}+ \mathbb{Z}\delta, \\
\Phi_a^+ &= ((\Phi^{\vee})^+ +\mathbb{N}\delta) \sqcup ((\Phi^{\vee})^-+\mathbb{N}^*\delta).
\end{split}
\end{align*}

There is a natural bijection between the roots of $\Phi_a^+$ and the hyperplanes of $\mathcal{H}$. This bijective correspondence is given by
\begin{align}\label{map}
\alpha^{\vee} + k\delta \longmapsto H_{\alpha,-k}.
\end{align}
In particular the reflection $s_{\alpha,k}$ acting in $V$ can be thought of as the reflection $s_{\alpha^{\vee}-k\delta}$ acting in $\widehat{V}$.

Let $\theta \in \Phi^{\vee}$. We denote by 
$$
\widetilde{\theta} = \left\{
                          	\begin{array}{rl}
                          	  \theta + \mathbb{N}\delta    & \text{if}~~ \theta \in (\Phi^{\vee})^+\\
 						  \theta + \mathbb{N}^*\delta & \text{if}~~ \theta \in (\Phi^{\vee})^-. \\
					    \end{array}
					    \right.
$$

For $\Gamma \subset \Phi^{\vee}$ we also denote $\widetilde{\Gamma} = \bigcup\limits_{\theta \in \Gamma}\widetilde{\theta}$.

Let $\Delta_1, \Delta_2 \subset \Delta$. For $X, Y \subset V$ we denote $\text{span}_{X}(Y):= \text{span}(Y) \cap X$. Then, we denote by $\Phi^+_{1,2}$ the set
$$
\Phi^+_{1,2}:=[(\Phi^{\vee})^+ \setminus \text{span}_{\Phi^{\vee}}(\Delta_1)] \sqcup [\text{span}_{\Phi^{\vee}}(\Delta_2)]
$$

and by $A_{1,2}$ the set 
\begin{align*}
A_{1,2} & := \{s_{\varepsilon}~|~\varepsilon \in \widetilde{\Phi^+_{1,2}} \} \\
            &  = \{ s_{\alpha^{\vee} +k\delta}~|~\alpha^{\vee} \in \Phi^+_{1,2},~k \in \mathbb{N} ~\text{if}~\alpha \in \Phi^+~ \text{and}~k \in \mathbb{N}^*~\text{if} ~\alpha \in \Phi^-      \} \\
            &\simeq \{s_{\alpha,k}~|~\alpha \in \Phi^+ \cap  (\Phi^+_{1,2})^{\vee} ,~k \in \mathbb{Z}\}.
\end{align*}

\medskip

\begin{definition}\label{def}
Let $x, y \in W_a$. We define (using the convention (\ref{conv2}))
\begin{itemize}
\item[1)]$
r_{1,2}^-(x) := \sum\limits_{\begin{subarray}{c}
 ~ ~~\alpha \in \Phi^{+}_{1,2} \\ 
  k(x, \alpha) \leq 0 
\end{subarray}}
k(x,\alpha)
~~~
\text{~~and~~}
~~~
r_{1,2}^+(x) := \sum\limits_{\begin{subarray}{c}
 ~ ~~\alpha \in \Phi^{+}_{1,2} \\ 
  k(x, \alpha) > 0 
\end{subarray}}
k(x,\alpha).
$
\item[2)] $
R^-_{1,2}(x,y) := r_{1,2}^-(yx)-r_{1,2}^-(y) ~~~ \text{~~and~~}~~~ R^+_{1,2}(x,y) := r_{1,2}^+(yx)-r_{1,2}^+(y).
$
\end{itemize}
\end{definition}

\medskip

\begin{theorem}[\cite{MD}, Theorem 1.15]\label{theo dy}
The set $$\{ \Phi^+_{i,j}~|~\Delta_i, \Delta_j \subset \Delta,~ \langle \Delta_i, \Delta_j \rangle =0\}$$ is a complete set of $W_a$-orbit representatives for the twisted conjugation on $\mathcal{B}(T)$.
\end{theorem}

Some explanations of this theorem can be found  in the preliminaries of \cite{wang2020reduced}.

\medskip

\section{Main results}\label{results}

The goal of this section is to give formulas about the affine twisted length function $\ell_A$ for specific $A \in \mathcal{P}(T)$. Proposition \ref{prop} expresses the length function $\ell_{A}$ with $A=\widetilde{(\Phi^{\vee})^+}$.  Finally, Theorem \ref{theo gen}  expresses the length $\ell_B$, for $B$ a biclosed subset of $T$, using the Shi coefficients.

\begin{proposition}\label{prop}
Let $A=\widetilde{(\Phi^{\vee})^+}$. Then $\ell_A(w) = \sum\limits_{\alpha \in \Phi^+}k(w^{-1},\alpha)$.
\end{proposition}

\begin{proof}
First of all, Proposition 3.4 of \cite{JYS1} tells us that  $\forall w \in W_a$ and $\forall \alpha \in \Phi$ we have 
\begin{align}\label{-1}
k(w^{-1}, \alpha) = -k(w,\overline{w}(\alpha))
\end{align}
 where $\overline{w}$ is the finite part of $w$. It is well known that for each $y \in W_a$, the hyperplanes corresponding to the reflections of $N(y)$ are exactly the hyperplanes between $A_y$ and $A_e$.
Via the correspondence (\ref{map}), the sets $\widetilde{(\Phi^{\vee})^+}$ and $\{H_{\alpha,k}~|~\alpha \in \Phi^+,~k \leq 0 \}$ are in bijection. Through this  identification, $N(w^{-1}) \cap A$ is the set of hyperplanes $H_{\alpha,k}$ such that $k \leq 0$ and $H_{\alpha,k}$ is between $A_{w^{-1}}$ and $A_e$. Furthermore, for a root $\alpha \in \Phi^+$ fixed, the number of $\alpha$-hyperplanes between $A_{w^{-1}}$ and $A_e$ is exactly $|k(w^{-1},\alpha)|$. 
Thus we have the formula
\begin{align}\label{ident}
|N(w^{-1}) \cap A|= \sum\limits_{\begin{subarray}{c}
 ~ ~~\alpha \in \Phi^{+} \\ 
  k(w^{-1},\alpha) \leq 0 
\end{subarray}}
|k(w^{-1},\alpha)|.
\end{align}

Then, the following computation gives the answer
\begin{align*}
\ell_{A}(w) & = \ell(w) - 2|N(w^{-1}) \cap A|  \stackrel{(\ref{length function})}{=} \ell(w^{-1}) - 2|N(w^{-1}) \cap A| \\ 
                  & \stackrel{(\ref{long})}{=} \sum\limits_{\alpha \in \Phi^+} |k(w^{-1},\alpha)|- 2|N(w^{-1}) \cap A| \\
                  & \stackrel{(\ref{-1}, \ref{ident})}{=} \sum\limits_{\alpha \in \Phi^+} |k(w,\overline{w}(\alpha))| - 2 \sum         	   \limits_{\begin{subarray}{c}
 ~ ~~\alpha \in \Phi^{+} \\ 
  k(w^{-1},\alpha) \leq 0 
\end{subarray}}
|k(w^{-1},\alpha)|  \\
             & = \sum\limits_{\alpha \in \Phi^+} |k(w,\overline{w}(\alpha))| + 2\sum         	   \limits_{\begin{subarray}{c}
 ~ ~~\alpha \in \Phi^{+} \\ 
  k(w^{-1},\alpha) \leq 0 
\end{subarray}}
k(w^{-1},\alpha) \\
& \stackrel{(\ref{-1})}{=}\sum\limits_{\alpha \in \Phi^+} |k(w,\overline{w}(\alpha))| - 2\sum         	   \limits_{\begin{subarray}{c}
 ~ ~~\alpha \in \Phi^{+} \\ 
  k(w, \overline{w}(\alpha)) \geq 0 
\end{subarray}} 
|k(w,\overline{w}(\alpha))| \\
& = \sum\limits_{\begin{subarray}{c}
 ~ ~~\alpha \in \Phi^{+} \\ 
  k(w, \overline{w}(\alpha)) \geq 0 
\end{subarray}}
|k(w,\overline{w}(\alpha))| + \sum\limits_{\begin{subarray}{c}
 ~ ~~\alpha \in \Phi^{+} \\ 
  k(w, \overline{w}(\alpha)) < 0 
\end{subarray}}
|k(w,\overline{w}(\alpha))| -2\sum\limits_{\begin{subarray}{c}
 ~ ~~\alpha \in \Phi^{+} \\ 
  k(w, \overline{w}(\alpha)) \geq 0 
\end{subarray}}
k(w,\overline{w}(\alpha)) \\
&= \sum\limits_{\begin{subarray}{c}
 ~ ~~\alpha \in \Phi^{+} \\ 
  k(w, \overline{w}(\alpha)) \geq 0 
\end{subarray}}
k(w,\overline{w}(\alpha)) - \sum\limits_{\begin{subarray}{c}
 ~ ~~\alpha \in \Phi^{+} \\ 
  k(w, \overline{w}(\alpha)) < 0 
\end{subarray}}
k(w,\overline{w}(\alpha)) -2\sum\limits_{\begin{subarray}{c}
 ~ ~~\alpha \in \Phi^{+} \\ 
  k(w, \overline{w}(\alpha)) \geq 0 
\end{subarray}}
k(w,\overline{w}(\alpha)) \\ 
& = - \sum\limits_{\begin{subarray}{c}
 ~ ~~\alpha \in \Phi^{+} \\ 
  k(w, \overline{w}(\alpha)) < 0 
\end{subarray}}
k(w,\overline{w}(\alpha)) -\sum\limits_{\begin{subarray}{c}
 ~ ~~\alpha \in \Phi^{+} \\ 
  k(w, \overline{w}(\alpha)) \geq 0 
\end{subarray}}
k(w,\overline{w}(\alpha)) \\ 
& = -\sum\limits_{\alpha \in \Phi^+} k(w,\overline{w}(\alpha)) \\
& \stackrel{(\ref{-1})}{=} \sum\limits_{\alpha \in \Phi^+} k(w^{-1},\alpha).
\end{align*}
\end{proof}

\begin{example} We give in Figure \ref{image} an element $w$ where we can see the formula of Proposition \ref{prop}. Indeed $\ell_A(w)=4-2|N(w^{-1})\cap A|= 4-2\cdot 4 = -4$.

\vspace*{\stretch{1}}
\begin{figure}[h!]
\includegraphics[scale=0.58]{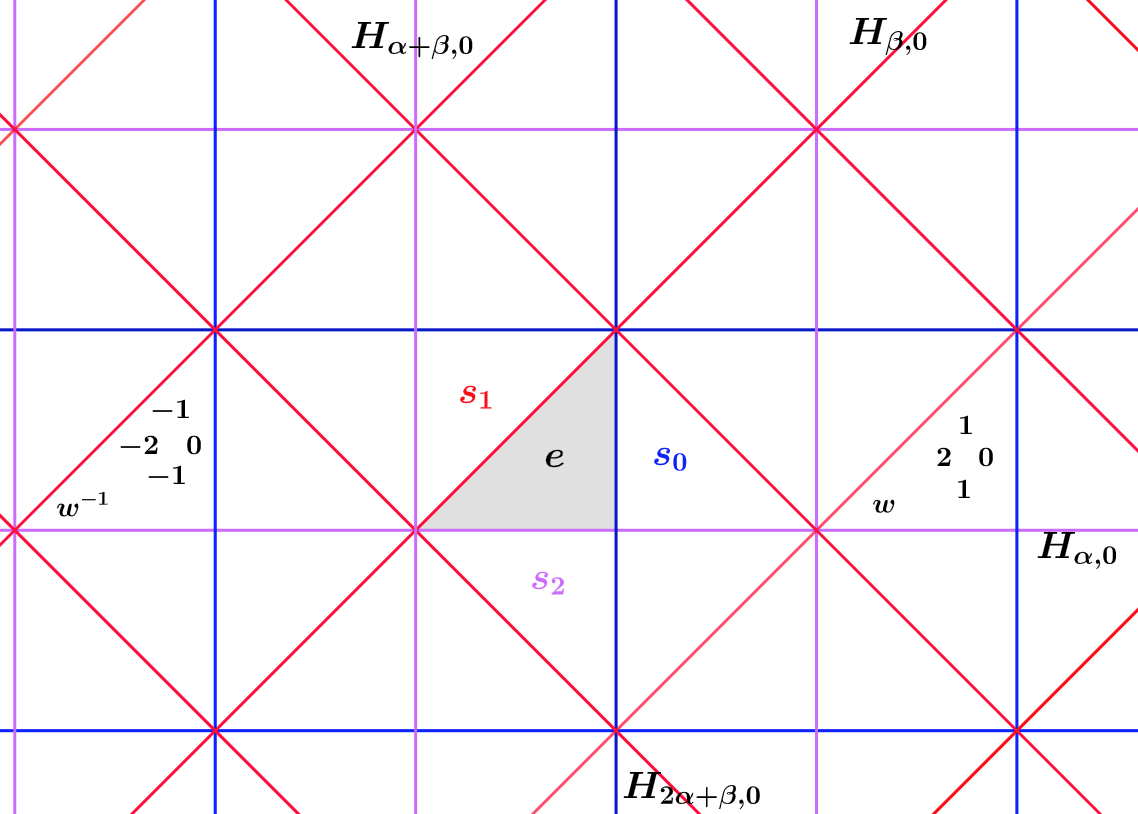} 
\caption{Application of Proposition \ref{prop} for $\Phi = B_2$.}
\label{image}
\end{figure}
\vspace*{\stretch{1}}
\end{example}

\begin{remark}
We also want to point out the following fact: in the formula of the standard length of $W_a$ (\ref{long}), the absolute value is involved, which most of the times makes the computations difficult. In Proposition \ref{propo long simple}, the computations of $\ell_A$ in terms of Shi coefficients make the absolute value disappears. 
This phenomenon is of interest in order to understand the set $N^1_A(w) = \{s_{\alpha,k} \in W_a~|~\ell_A(s_{\alpha,k}w) = \ell_A(w)-1\}$. It is likely that for any $B \in \mathcal{B}(T)$, a better expression of $\ell_B$ would  provide a better comprehension of $N^1_B(w)$.

 When $A=\emptyset$, the set $N^1(w):=N^1_A(w)$ is strongly related to the automatic and geometric structure of Coxeter groups, namely to a conjecture due to M. Dyer and C. Hohlweg relating low elements and Shi regions \cite{SRLE}.
\end{remark}

\begin{lemma} \label{propo long simple}
Let $x \in W_a$ and $\Delta_1, \Delta_2 \subset \Delta$. Then we have the formula
$$\ell_{A_{1,2}}(x) =  \ell(x) +2r_{1,2}^-(x^{-1}) - 2r_{1,2}^+(x^{-1}).$$
\end{lemma}

\begin{proof}
By definition $\ell_{A_{1,2}}(x) =   \ell(x) - 2|N(x^{-1}) \cap A_{1,2}|$. Using the bijection (\ref{map}), we see that $A_{1,2}$ is in bijection with $\{H_{\alpha,k}~|~\alpha \in \Phi_{1,2}^+,~ k \leq 0 \} \sqcup  \{H_{\alpha,k}~|~\alpha \in \Phi_{1,2}^+, ~k > 0 \}$. It follows that 
\begin{align*}
|N(x^{-1}) \cap A_{1,2}| & = \sum\limits_{\begin{subarray}{c}
 ~ ~~\alpha \in \Phi^{+}_{1,2} \\ 
  k(x^{-1}, \alpha) \leq 0 
\end{subarray}}
|k(x^{-1},\alpha)| + \sum\limits_{\begin{subarray}{c}
 ~ ~~\alpha \in \Phi^{+}_{1,2} \\ 
  k(x^{-1}, \alpha) > 0 
\end{subarray}}
|k(x^{-1},\alpha)| \\
& = -\sum\limits_{\begin{subarray}{c}
 ~ ~~\alpha \in \Phi^{+}_{1,2} \\ 
  k(x^{-1}, \alpha) \leq 0 
\end{subarray}}
k(x^{-1},\alpha) + \sum\limits_{\begin{subarray}{c}
 ~ ~~\alpha \in \Phi^{+}_{1,2} \\ 
  k(x^{-1}, \alpha) > 0 
\end{subarray}}
k(x^{-1},\alpha) \\
& = -r_{1,2}^-(x^{-1}) + r_{1,2}^+(x^{-1}).
\end{align*}

The result follows.
\end{proof}

\begin{theorem}\label{theo gen}
Let $B \in \mathcal{B}(T)$. Then there exist $\Delta_1, \Delta_2 \subset \Delta$ (two pairwise orthogonal sets) and $y \in W_a$ such that for all $x \in W_a$ we have
 $$
 \ell_{B}(x) = \ell(xy) - \ell(y) + 2R_{1,2}^-(x^{-1},y^{-1}) -2R_{1,2}^+(x^{-1},y^{-1}).$$
\end{theorem}

\begin{proof}

By Theorem \ref{theo dy} we know that $B=y\cdot A_{1,2}$ for some $\Delta_1, \Delta_2 \subset \Delta$ and for some $y \in W_a$. By (\ref{length}) we know that $\ell_{A_{1,2}}(xy) = \ell_{A_{1,2}}(y) +\ell_{y\cdot A_{1,2}}(x)$. In particular we have $\ell_{B}(x) = \ell_{A_{1,2}}(xy) - \ell_{A_{1,2}}(y)$. Then, Lemma \ref{propo long simple} implies that
\begin{align*}
 \ell_{B}(x)  & =  \ell_{A_{1,2}}(xy) - \ell_{A_{1,2}}(y) \\
                 & = \big[ \ell(xy) + 2r_{1,2}^-((xy)^{-1}) - 2r_{1,2}^+((xy)^{-1}) \big]-\big[ \ell(y) +2r_{1,2}^-(y^{-1}) - 2r_{1,2}^+(y^{-1}) \big] \\
                 & = \big[ \ell(xy) + 2r_{1,2}^-(y^{-1}x^{-1}) - 2r_{1,2}^+((y^{-1}x^{-1}) \big]-\big[ \ell(y) +2r_{1,2}^-(y^{-1}) - 2r_{1,2}^+(y^{-1}) \big] \\
                 & = \ell(xy) - \ell(y) + 2\big[r_{1,2}^-(y^{-1}x^{-1})- r_{1,2}^-(y^{-1})\big] -2 \big[r_{1,2}^+(y^{-1}x^{-1})- r_{1,2}^+(y^{-1}) \big]    \\
                 & = \ell(xy) - \ell(y) + 2R_{1,2}^-(x^{-1},y^{-1}) -2R_{1,2}^+(x^{-1},y^{-1}).          
\end{align*}
\end{proof}

\medskip

\begin{remark}
In particular, if $xy$ is a reduced expression it follows that $\ell_{B}(x) = \ell(x) + 2R_{1,2}^-(x^{-1},y^{-1}) -2R_{1,2}^+(x^{-1},y^{-1})$. We also see  that the formula in Proposition \ref{propo long simple} can be obtained from Theorem \ref{theo gen}  with $y =e$.
\end{remark}

\textbf{Acknowledgements}. I am grateful to Matthew Dyer who suggested the idea investigated in this paper and who provided many useful references. I thank the referee for some useful suggestions concerning the presentation of this note.

\bibliographystyle{plain}
\bibliography{length_function.bib}

\end{document}